\documentclass[preprint,3pt]{elsarticle}

\usepackage{amssymb}
\usepackage{amsthm,amssymb,amsmath,extarrows}
\usepackage{pdfpages}
\usepackage{graphicx}
\usepackage{geometry}
\usepackage{url}
\usepackage{enumerate}
\usepackage{extarrows}
\theoremstyle{plain} \newtheorem{thm}{\bf Theorem}[section]
 \newtheorem{lem}[thm]{\bf Lemma}
  \theoremstyle{remark} \newtheorem{remark}[thm]{\bf Remark}

\newcommand{\Pro}{\mathbf{P}}
\newcommand{\Es}{{\mbox{\bf E}}}

\newcommand{\MLE}{\hat{\vartheta}}

\newcommand{\Rg}       {{\hbox{I\kern-.22em\hbox{R}}}}
\newcommand{\Pg}       {{\hbox{I\kern-.22em\hbox{P}}}}
\newcommand{\Eg}       {{\hbox{I\kern-.22em\hbox{E}}}}

\newcommand{\tr}   {\mbox{ \rm{tr} }}

\definecolor{lw}{RGB}{0,0,255}

\journal{journal}

\begin{document}

\begin{frontmatter}



\title{A Note On Inference for the Mixed Fractional Ornstein-Uhlenbeck  Process with Drift}


\author{Cai Chunhao \\Shanghai University of Finance and Economics   
\\caichunhao@mail.shufe.edu.cn\\
Zhang Min\\Shanghai University of Finance and Economics\\zm\_hardworker@163.com}

\begin{abstract}
This paper is devoted to parameter estimation of the mixed fractional Ornstein-Uhlenbeck process with a drift. Large sample asymptotical properties of the Maximum Likelihood Estimator is deduced using the Laplace transform computations or the Cameron-Martin formula with extra part from \cite{CK19}.\end{abstract}

\begin{keyword}

mixed fractional Brownian motion,\, fundamental martingale,\, Laplace Transform,\, optimal input

\vspace{1mm}

\end{keyword}

\end{frontmatter}

\section{Introduction}
The drift parameter estimation of the Ornstein-Uhlenbeck process has been paid more and more attention in the past decades. These years, researchers not only considered the process with standard Brownian motion or Levy process but also the fractional case (see \textit{e.g.} \cite{Yuri} \cite{Klep02a}). The MLE and its large deviation of the mixed fractional case has been studied by Chigansky et \textit{a.l.} \cite{CK19}, \cite{Maru}. In this paper we will consider still the MLE of the drift parameter but with an extra part. 

Let us define $X=(X_t, 0\leq t\leq T)$ a real-valued process, representing the observation,  which is governed by:
\begin{equation}\label{model}
dX_t=-\vartheta X_tdt+u(t)dt+d\xi_t,\, t\in [0,T],\, X_0=0
\end{equation}
where $\xi=(\xi_t,\, 0\leq t\leq T)$ is a mixed fractional fractional  Brownian motion (mfBm for short) which is defined by $\xi_t=W_t+B_t^H$ , here $W=(W_t,\, 0\leq t\leq T)$ and  $B^H=(B_t^H,\, 0\leq t \leq T$ are independent standard Brownian motion and fractional Brownian motion with $H\in (0,1),\, H\neq 1/2$.   First of all, we will consider $u(t)$ is known constant. In fact, this problem has been considered in \cite{CHX20} but here we will use the Cameron-Martin formula which will be the key method for this type problem to reprove the result. After that we will consider the problem of experiment design. 

In the statistical aspect, the classical approach for experiment design consists on a two-step procedure: maximize the Fisher information under energy constraint of the input and find an adaptive estimation procedure. Ovseevich et \textit{a.l.} \cite{Oss} has first consider this type problem for the diffusion equation with continuous observation.  When the kernel in \cite{Oss} is not with explicit formula in the fractional diffusion case, Brouste et \textit{a.l.} \cite{BKP10, BKP12} deduce the lower bound and upper bound with the method of spectral gap and solve the same problem. Base on this method, Brouste and Cai  \cite{BC13} have extended the result to the partially observed fractional Ornestein-Uhlenbeck process, in this work the asymptotical normality has been demonstrated with  linear filtering of Gaussian processes and Laplace Transform presented in \cite{Liptser, BK10, Klep01, Klep02a, Klep02b}. These previous work, the common point is that: the optimal input does not depend on the unknown parameter and maximum likelihood estimator can be found directly from the likelihood equation.  The one-step estimator will be used following the Newton-Raphson method and this work was introduced by Cai and LV \cite{CL20}.

For a fixed value of parameter $\vartheta$, let $\mathbf{P}_{\vartheta}^T$ denote the probability measure, induced by $X^T$ on the function space $\mathcal{C}_{[0,\,T]}$ and let $\mathcal{F}_t^X$ be the nature filtration of $X$, $\mathcal{F}_t^X=\sigma(X_s,\, 0\leq s\leq t)$. Let $\mathcal{L}(\vartheta,\, X^T)$ be the likelihood, {\it i.e.}  the Radon-Nikodym derivative of $\mathbf{P}_{\vartheta}^T$,  restricted to $\mathcal{F}_T^Y$ with respect to some reference measure on $\mathcal{C}_{[0,\,T]}$. In this setting, Fisher information stands for 
$$
\mathcal{I}_T(\vartheta,\, u)=-\mathbf{E}_{\vartheta}\frac{\partial^2}{\partial \vartheta^2} \ln \mathcal{L}(\vartheta,\,X^T).
$$
Let us denote $\mathcal{U}_T$ some functional space of controls, that is defined by Eqs. \eqref{eq:u}
and \eqref{space v}. Let us therefore note
\begin{equation}\label{eq: J max}
\mathcal{J}_T(\vartheta)=\sup_{u\in \mathcal{U}_T}\mathcal{I}_T(\vartheta,\,u).
\end{equation}
our main goal is to find estimator $\overline{\vartheta}_T$ of the parameter $\vartheta$ which is asymptotically efficient in the sense that, for any compact $\mathbb{K}\in \mathbb{R}^{+}_{*}=\{\vartheta \in \mathbb{R}, \vartheta>0\}$ ,
\begin{equation} \label{asymp}
\sup_{\vartheta \in \mathbb{K}} \mathcal{J}_T(\vartheta) \mathbf{E}_{\vartheta} \left( \overline{\vartheta}_T - \vartheta \right)^2 =1+o(1) \,,
\end{equation}
as $T \rightarrow \infty$. 

As the optimal input does not depend on $\vartheta$ (see Proposition~\ref{thm1}),  a possible candidate is the Maximum Likelihood Estimator (MLE) $\hat{\vartheta}_T$, defined as the maximizer of
the likelihood:
\begin{equation}\nonumber
\hat{\vartheta}_T = \mbox{arg}\max_{\vartheta>0} {\cal L}(\vartheta,X^{T}).
\end{equation}
We want to find the asymptotical normality of the MLE of $\vartheta$.

The interest to mixed fractional Brownian motion was triggered by Cheridito\cite{cheridito01}. The resent works of Cai, Chigansky, Kleptsyna and Marushkevych  (\cite{CCK16}  \cite{Maru} \cite{CK18} \cite{CK19}) present a great value for the purpose of this paper.  The process $\xi_t$ satisfies a number of curious properties with applications in mathematical finance, see \cite{Christian}. In particular, as shown in \cite{cheridito01, cheridito03} , it is a semimartingale if and only if $H \in \{\frac{1}{2}\}\bigcup (\frac{3}{4},\,1]$ and the measure $\mu^{\xi}$ induced by $\xi$ on the space of continuous functions on $[0,\,T]$, is equivalent to the standard Wiener measure  $\mu^B$ for $H>\frac{3}{4}$. On the other hand,  $\mu^{\xi}$ and $\mu^{B^H}$ are equivalent if and only if $H<\frac{1}{4}$.

The paper falls into four parts. In the second part, we present some main results of this paper and the third part will contribute to the proofs of the main results. Some Lemmas will be given in Appendix.

\section{Main Results}

\subsection{Transformation of the model}

Even if the mixed fractional Brownian motion $\xi$ is a semimartingale when $H>\frac{3}{4}$, it is hard to write the likelihood function directly. We will try to transform our model with the fundamental martingale in \cite{CCK16} and get the explicit representation of the likelihood function.  In what follows, all random variables and processes are defined on a given stochastic basis $(\Omega,\mathcal{F},(\mathcal{F}_t)_{t \geq 0},\mathbf{P})$ satisfying the usual conditions and processes are $(\mathcal{F}_t)-$ adapted. Moreover the {\it natural filtration} of a process is understood as the $\mathbf{P}$-completion of the filtration generated by this process. 

From the canonical innovation representation in \cite{CCK16},  the fundamental martingale is defined as $M_t=\mathbf{E}(B_t|\mathcal{F}_t^{\xi})$,  $t\in [0,\,T]$, then for $H>1/2$ this martingale satisfies 
\begin{equation}\label{def:mar}
M_t=\int_0^t g(s,t)d\xi_s ,\,\,\,\,\,\,\,\,\,    \langle M\rangle_t=\int_0^t g(s,t)ds
\end{equation}
where $g(s,t)$ is the solution of the integro-differential equation
\begin{equation}\label{int-diff}
g(s,t)+H\frac{d}{ds} \int_0^tg(r,t)|r-s|^{2H-1} \textrm{sign}(s-r)dr=1,\, 0<s\leq t \leq T
\end{equation}

Following from \cite{CCK16}, let us introduce a process $Z=(Z_t,\, 0\leq t\leq T)$ the fundamental semimartingale associated to $X$, defined as 
$$
Z_t=\int_0^t g(s,t) dX_s.
$$
Note that $X$ can be represented as $X_t=\int_0^t \hat{g}(s,t)dZ_s$ where 
\begin{equation}\label{hat g}
\hat{g}(s,t)=1-\frac{d}{d\langle M\rangle_s}\int_0^t g(r,s)dr
\end{equation}
for $0\leq s\leq t$ and there for the nature filtration of X and Z coincide.  Moreover, we have the following representations:
\begin{equation}\label{eq:Z}
dZ_t=-\vartheta Q_td\langle M\rangle_t+v(t)d\langle M\rangle_t+dM_t,
\end{equation}
where 
\begin{equation}\label{eq:Q and v}
Q_t=\frac{d}{d\langle M\rangle_t}\int_0^t g(s,t)X_sds,\, v(t)=\frac{d}{d\langle M\rangle_t}\int_0^t g(s,t)u(s)ds.
\end{equation}
\subsection{MLE of with the Constant Drift}
Let us consider $u(t)=\alpha$ a constant not $0$.  In this case we will denote the processes $X$, $Z$, $Q$ by $X^{\alpha}$, $Z^{\alpha}$ and $Q^{\alpha}$ and it is not hart to find that the MLE of the unknown parameter $\vartheta$ is 
\begin{equation}\label{eq: tilde beta}
\hat{\vartheta}^{\alpha}_T=\frac{\int_0^T \alpha Q^{\alpha}_t d\langle M\rangle_t-\int_0^T Q^{\alpha}_tdZ^{\alpha}_t}{\int_0^T (Q_t^{\alpha})^2d\langle M\rangle_t}
\end{equation}
where 
\begin{equation}\label{eq: Zt alpha}
dZ^{\alpha}_t=(\alpha-\vartheta Q^{\alpha}_t)d\langle M\rangle_t+dM_t,\, t\in [0,T].
\end{equation}
From \cite{CHX20} the estimation error can be presented by 
\begin{equation}\label{eq: tilde beta error}
\hat{\vartheta}^{\alpha}_T-\vartheta=\frac{\int_0^T Q^{\alpha}_t dM_t}{\int_0^T (Q_t^{\alpha})^2d\langle M\rangle_t}.
\end{equation}
and we have:
\begin{thm}\label{th Laplace Transform}
For $H>1/2$, 
$$
\sqrt{T}(\hat{\vartheta}^{\alpha}_T-\vartheta)\xrightarrow{d}\mathcal{N}(0,2\vartheta)
$$
and for $H<1/2$,
$$
\sqrt{T} (\hat{\vartheta}^{\alpha}_T-\vartheta)\xrightarrow{d} \mathcal{N}\left(0, \frac{2\vartheta^2}{2\alpha^2+\vartheta}\right)
$$
\end{thm}
In fact we have proved this result in \cite{CHX20} but here we want to use the general Laplace Transform (Cameron-Martin formula) which will also be presented in the following optimal input case.  We can find the proof in Section \ref{sec proof}. On the other hand we can also complete it same as the following theorem \ref{thm2}.

\subsection{Optimal Input Case}
\subsubsection{Two dimensional observation}
First of all, let us define the space of control for $v(t)$:
\begin{equation}\label{space v}
\mathcal{V}_T=\left\{h \Big{|}     \frac{1}{T}\int_0^T|v(t)|^2d\langle M\rangle_t \leq 1 \right\}.
\end{equation}
Remark that with \eqref{eq:Q and v} the following relationship between control $u$ and its transformation $v$ holds:
\begin{equation}\label{eq:u}
u(t)=\frac{d}{dt}\int_0^t \hat{g}(t,s)v(s)d\langle M\rangle_s
\end{equation}
we can set the admissible control as $\mathcal{U}_T=\{u|v\in \mathcal{V}_T\}$. Note that these set are non-empty. 

From \cite{CK19}, we know $Q_t=\int_0^t \psi(s,t)dZ_s$ where
\begin{equation}\label{psi}
\psi(s,t)=\frac{1}{2}\left(\frac{dt}{d\langle M\rangle_t}   +\frac{ds}{d\langle M\rangle_s}\right).
\end{equation}
Moreover, $Q_t=\frac{1}{2}\ell(t)^*\zeta_t$, where $\ell(t)=\left(\begin{array}{c} \psi(t,t) \\  1\end{array}\right)$, $*$ standing for the transposition and $\zeta=(\zeta_t,\,t\geq 0)$ is the solution of the stochastic differential equation
\begin{equation}\label{eq:zeta}
d\zeta_t=-\frac{\vartheta}{2}A(t)\zeta_td\langle M\rangle_t+b(t)v(t)d\langle M\rangle_t+b(t)dM_t, \zeta_0=\mathbf{0}_{2\times 1}, 
\end{equation}
with
\begin{equation}\label{matrix A b}
A(t)=\left(\begin{array}{cc} \psi(t,t) &  1 \\  \psi^2(t,t)& \psi(t,t) \end{array}\right),\,
b(t)=\left(\begin{array}{c} 1 \\ \psi(t,t) \end{array}\right).
\end{equation}

\subsubsection{Likelihood function and Fisher information}
The classical Girsanov theorem gives 
\begin{equation}\label{eq: likelihood function without}
\mathcal{L}(\vartheta,\, Z^T)=\mathbf{E}_{\vartheta}\exp \left\{-\int_0^T(-\vartheta Q_t+v(t))dZ_t- \frac{1}{2}\int_0^T(-\vartheta Q_t+v(t))^2d\langle M\rangle_t\right\},
\end{equation}
then the Fisher information stands for 
\begin{eqnarray*}
  \mathcal{I}_T(\vartheta, v) &=& - \mathbf{E}_{\vartheta}\frac{\partial ^2}{\partial \vartheta^2}\ln\mathcal{L}(\vartheta, Z^{T}) \\
   &=& \frac{1}{4}\mathbf{E}_{\vartheta}\int_0^T\left(\ell(t)^*\zeta_t\right)^2d\langle M\rangle.
\end{eqnarray*}
Then we have the following results for the optimal input:
\begin{thm}\label{thm1}
The asymptotic optimal input in the class of controls $\mathcal{U}_T$ is $u_{opt}(t)=\frac{d}{dt}\int_0^t\hat{g}(s,t)\psi(s,s)d\langle M\rangle_s$ where $\hat{g}(s,t)$,  $\psi(s,t)$, $\langle M\rangle_t$ are defined in \eqref{def:mar}, \eqref{hat g}, \eqref{psi}. Moreover,
$$
\lim_{T \to + \infty}\frac{\mathcal{J}_T(\vartheta)}{T}=\mathcal{I}(\vartheta),
$$
where
\begin{equation}\label{fisher}
\mathcal{I}(\vartheta)=\frac{1}{2\vartheta}+\frac{1}{\vartheta^2}.
\end{equation}
The $\mathcal{J}_T(\vartheta)$ is defined in \eqref{eq: J max}.
\end{thm}
 
\subsubsection{Asymptotical Normality of The MLE}
From the theorem \ref{thm1}, we can see that the optimal input $u_{opt}(t)$ does not depend on the unknown parameter $\vartheta$, we can easily obtain the estimator error of the MLE of the  $\hat{\vartheta}_T$:
\begin{equation}\label{eq: MLE error}
\hat{\vartheta}_T-\vartheta=\frac{\int_0^T Q_t dM_t}{\int_0^T Q_t^2d\langle M\rangle_t}.
\end{equation} 
Then, the MLE reaches efficiency and we deduce its large sample asymptotic properties:
\begin{thm}\label{thm2} 
The MLE is uniformly consistent on compacts $\
K \subset \mathbb{R}_*^+$, {\it i.e.} for any $\nu >0$,
\begin{equation}\nonumber
\lim_{T\rightarrow \infty} \sup_{\vartheta \in \mathbb{K}}\Pro_\vartheta^{T} \left\{\left|\MLE_T-\vartheta \right|> \nu \right\}=0\,,
\end{equation}
uniformly on compacts asymptotically normal: as $T$ tends to $+\infty$,
\begin{equation}\label{eq:asympnormal}\nonumber
\lim_{T\rightarrow\infty}  \sup_{\vartheta \in \mathbb{K}} \left| \Es_\vartheta f\left( \sqrt{T} \left( \hat{\vartheta}_T - \vartheta\right) \right) - \Es f(\eta) \right| =0\quad \forall f \in {\cal C}_b
\end{equation}
and $\xi$ is a zero mean Gaussian random variable of variance ${\cal I}(\vartheta)^{-1}$ (see  \eqref{fisher} for the explicit value)
{\bf which does not depend on $H$}
 and we have the uniform on $\vartheta\in \mathbb{K}$ convergence of the moments: for any $p>0$,
\begin{equation}\nonumber
\lim_{T\rightarrow \infty}  \sup_{\vartheta \in \mathbb{K}} \left| \Es_\vartheta \left|\sqrt{T} \left( \MLE_T-\vartheta\right) \right |^p-  \Es \left|\eta\right|^p \right| =0.
\end{equation}
Finally, the MLE is efficient in the sense of \eqref{asymp}.
\end{thm}

\begin{thm}\label{thm strong consistency}
The MLE $\hat{\vartheta}_T$ is strong consistency that is 
$$
\hat{\vartheta}_T\xrightarrow {a.s.} \vartheta,\,\,\,\, T\rightarrow \infty.
$$
\end{thm}
\begin{remark}
This strong consistency is also true for $u(t)=\alpha$.
\end{remark}

\section{Proofs of Main Results}\label{sec proof}

\subsection{Proof of Theorem \ref{th Laplace Transform}}
In order to get the property of asymptotical normality of this estimation error, we always apply the central limit theorem of martingale (see \cite{HH80}) to try to find the limit of the denominator in probability. Different from the development with the mixed O-U process with $\alpha=0$, we compute the the limit of Laplace Transform $\mathcal{L}^{\alpha}_T(\mu)$:
$$
\mathcal{L}^{\alpha}_T(\mu)=\mathbf{E}\exp\left(-\frac{\mu}{T}\int_0^T (Q_t^{\alpha})^2d\langle M\rangle_t\right)
$$
The result is obvious from Lemma \ref{lem Laplace Transform} with the limit of this Laplace Transform.

\subsection{Proof of Theorem \ref{thm1}}
We will compute the Fisher information with the same method in \cite{BC13}, that is to separate the Fisher information into two parts, on into the control, the other without, we focus on the following decomposition:
\begin{eqnarray}\label{eq:separ}
  \mathcal{I}_T(\vartheta,v) &=&\frac{1}{4}\mathbf{E}_{\vartheta}\left\{\int_0^T(\ell(t)^*\zeta_t-\mathbf{E}_{\vartheta}\ell(t)^*\zeta_t+\mathbf{E}_{\vartheta}\ell(t)^*\zeta_t)^2\right\} \nonumber\\
    &=& \mathcal{I}_{1,T}(\vartheta,v)+\mathcal{I}_{2,T}(\vartheta,v)
\end{eqnarray}
where
\begin{equation}\label{I 1}
\mathcal{I}_{1,T}(\vartheta,v)=\frac{1}{4}\int_0^T \mathbf{E}_{\vartheta}(\ell(t)^*\zeta_t-\mathbf{E}_{\vartheta}\ell(t)^*\zeta_t)^2\langle M\rangle_t
\end{equation}
and
\begin{equation}\label{I 2}
\mathcal{I}_{2,T}(\vartheta,v)=\frac{1}{4}\int_0^T (\ell(t)^*\mathbf{E}_{\vartheta}\zeta_t)^2d\langle M\rangle_t.
\end{equation}
The deterministic function $(\mathcal{P}(t)=\mathbf{E}_{\vartheta}\zeta_t,\, t\geq 0)$ satisfies the following equation:
\begin{equation}\label{eq: P}
\frac{d\mathcal{P}(t)}{d\langle M\rangle_t}=-\frac{1}{2}\vartheta A(t)\mathcal{P}(t)+b(t)v(t), \mathcal{P}(0)=\mathbf{0}_{2\times 1},
\end{equation}
at the same time the process $\overline{P}=(\overline{P}_t=\zeta_t-\mathbf{E}_{\vartheta}\zeta_t,\,t\geq 0)$ satisfies the following stochastic equation:
$$
d\overline{P}_t=-\frac{1}{2}\vartheta A(t)\overline{P}_td\langle M\rangle_t+b(t)dM_t, 
$$
which is just the $\zeta_t$ with $v(t)=0$ which can be found in \cite{CK19}.

With the technical separation of \eqref{eq:separ} and the precedent remarks, we have
$$
\mathcal{J}_T(\vartheta)=\mathcal{I}_{1,T}(\vartheta)+\mathcal{J}_{2,T}(\vartheta),
$$
where 
$$
\mathcal{J}_{2,T}(\vartheta)=\sup_{v\in \mathcal{V}_T}\mathcal{I}_{2,T}(\vartheta,\,v).
$$
From \cite{CK19}, we know 
$$
\lim_{T\to \infty}\frac{\mathcal{I}_{1,T}(\vartheta)}{T}=\frac{1}{2\vartheta},
$$
so we just need to check that  $\lim\limits_{T\to \infty}\frac{\mathcal{J}_{2,T}(\vartheta)}{T}=\frac{1}{\vartheta^2}$. From \eqref{eq: P}, we get
\begin{equation}\label{eq: P2}
\mathcal{P}(t)=\varphi(t)\int_0^t\varphi^{-1}(s)b(s)v(s)d\langle M\rangle_s,
\end{equation}
where $\varphi(t)$ is the matrix defined by
\begin{equation}\label{varphi}
\frac{d\varphi(t)}{d\langle M\rangle_t}=-\frac{\vartheta}{2}A(t)\varphi(t),\, \varphi(0)=\mathbf{Id}_{2\times 2}
\end{equation}
with $\mathbf{Id}_{2\times 2}$ the $2\times 2$ identity matrix. Substituting into \eqref{I 2}, we get
\begin{equation}\label{I 2 2}
\mathcal{I}_{1,T}(\vartheta,\,v)=\int_0^T \int_0^T\mathcal{K}_T(s,\,\sigma)\frac{1}{\sqrt{\psi(s,\,s)}}v(s)\frac{1}{\sqrt{\psi(\sigma,\,\sigma)}}v(\sigma)dsd\sigma,
\end{equation}
where the operator 
\begin{equation}\label{operator K}
\mathcal{K}_T(s,\,\sigma)=\int_{\max(s,\,\sigma)}^T\mathcal{G}(t,s)\mathcal{G}(t,\,\sigma)dt
\end{equation}
and
\begin{equation}\label{eq: G}
\mathcal{G}(t,\,\sigma)=\frac{1}{2}\left(\frac{1}{\sqrt{\psi(t,t)}}\ell(t)^*\varphi(t)\varphi^{-1}(\sigma)b(\sigma)\frac{1}{\sqrt{\psi(\sigma,\,\sigma)}}\right).
\end{equation}
Then
\begin{eqnarray}\label{eq:J2}
  \mathcal{J}_{2,T}(\vartheta) &=& T \sup_{\widetilde{v} \in L^2[0,T], \|\widetilde{v}\|\leq 1}\int_0^T\int_0^T\mathcal{K}_T(s,\sigma)\widetilde{v}(s)\widetilde{v}(\sigma)dsd\sigma ,\nonumber \\
  &=& T \sup_{\widetilde{v} \in L^2[0,T],\|\widetilde{v}\|\leq 1}(\mathcal{K}_T\widetilde{v},\widetilde{v})
\end{eqnarray}
where $\widetilde{v}(s)=\frac{v(s)}{\sqrt{T}}\frac{1}{\sqrt{\psi(t,t)}}$ and $\|\bullet\|$ stands for the usual norm in $L^2[0,\,T]$. Thus, Lemma \ref{lemma1} completes our proof.

\subsection{Proof of Theorem \ref{thm2}}
Taking $v_{opt}(t)=\sqrt{\psi(t,t)}$ into the equation  \eqref{eq: likelihood function without}, then the likelihood function is 
$$
\mathcal{L}(\vartheta,\, Z^T)=\mathbf{E}_{\vartheta}\exp \left\{-\int_0^T(-\vartheta Q_t+v_{opt}(t))dZ_t- \frac{1}{2}\int_0^T(-\vartheta Q_t+v_{opt}(t))^2d\langle M\rangle_t\right\},
$$
then the maximum likelihood estimator (MLE) will be
\begin{equation}\label{rep MLE}
\hat{\vartheta}_T=\frac{\int_0^Tv_{opt}(t)Q_td\langle M\rangle_t-\int_0^TQ_tdZ_t}{\int_0^TQ_t^2d\langle M\rangle_t}
\end{equation} 
and the estimation error has the form
\begin{equation}\label{error}
\hat{\vartheta}_T-\vartheta=-\frac{\int_0^TQ_tdM_t}{\int_0^TQ_t^2d\langle M\rangle_t},
\end{equation}
just take attention that here the $Q_t $ will be with the relationship with $v_{opt}(t)$. Because $\int_0^t Q_sdMs,\, 0\leq t \leq T$ is a martingale and $\int_0^t Q_s^2d\langle M\rangle_s$ is its quadratic variation, In order to prove the Theorem \ref{thm2}, we only need to check the Laplace Transform of the quadratic variation and Lemma \ref{Lapalace asymptotical} achieves the proof.

\subsection{Proof of Theorem \ref{thm strong consistency}}
With the law of large numbers, in order to obtain the strong consistency of $\vartheta$, we only need to prove that 
\begin{equation}\label{eq: to infty mar}
\lim_{T\rightarrow \infty}\int_0^T Q_t^2d\langle M\rangle_t=+\infty
\end{equation}
or there exists a positive constant $\mu$ such that the limit of the Laplace Transform 
$$
\lim_{T\rightarrow \infty}\mathbf{E}\exp\left(-\mu \int_0^T Q_t^2d\langle M\rangle_t    \right)=0.
$$
In Lemma \ref{Laplace log} if we take a big enough $\mu>0$ such that the limit is negative (the $\mu$ can be easily found), then the equation \eqref{eq: to infty mar} is directly from this Lemma which implies the strong consistency.

\section{Appendix}
\begin{lem}\label{lemma1}
For the kernel $\mathcal{K}_T(s,\sigma)$  defined in equation \eqref{eq:J2} 
\begin{equation}\label{lim operator}
\lim_{T\rightarrow \infty}\sup_{\widetilde{v}\in L^2[0,T],\|\widetilde{v}\|\leq1}(\mathcal{K}_T\widetilde{v},\widetilde{v})=\frac{1}{\vartheta^2}
\end{equation}
with an optimal input $v_{opt}(t)=\sqrt{\psi(t,t)}$
\end{lem}
\begin{proof}
When we take $v(t)=v_{opt}(t)=\sqrt{\psi(t,t)}$, then 
$$
\frac{d\mathcal{P}(t)}{d\langle M\rangle_t}=-\frac{1}{2}\vartheta A(t)\mathcal{P}(t)+b(t)v_{opt}(t), \mathcal{P}(0)=\mathbf{0}_{2\times 1}.
$$
Because for $H>1/2$, $\frac{d\langle M\rangle_t}{dt}=g^2(t,t)$. From \cite{CK19} 
$$
\langle M\rangle_T\sim T^{2-2H}\lambda_H^{-1},\, T\rightarrow \infty,\, \lambda_H=\frac{2H\Gamma(3-2H)\Gamma(H+1/2)}{\Gamma(3/2-H)}.
$$
then with the calculus of \cite{BKP12} we can easily obtain
\begin{equation}\label{eq lim H big}
\lim_{T\to \infty}\frac{1}{4T}\int_0^T(\ell(t)^*\mathcal{P}(t))^2d\langle M\rangle_t=\frac{1}{\vartheta^2}.
\end{equation}
On the other hand for $H<1/2$ we have 
$$
\lim_{T\rightarrow \infty} \frac{\langle M\rangle_T}{T}=1
$$
and we can also easily obtain the result of \eqref{eq lim H big}, that is to say the lower bound at least will be $\frac{1}{\vartheta^2}$.

Now we will try to find the upper bound. Let us introduce the Gaussian process $(\xi_t,\, 0\leq t\leq T)$
$$
\xi_t=\left(\frac{1}{\sqrt{\psi(\sigma,\sigma)}}\ell(\sigma)^*\varphi(\sigma)\odot dW_{\sigma}\right)\varphi^{-1}(t),\, \xi_T=0
$$
where $\left(W_\sigma, \sigma \geq 0 \right)$ is a Wiener process and  $\odot$ denotes the It\^{o} backward integral (see \cite{Roz}). It is worth emphasizing that
$$
\mathcal{K}_T(s,\sigma)=\frac{1}{4}\mathbf{E}\left(\xi_s b(s)\frac{1}{\sqrt{\psi(s,s)}}\xi_{\sigma}b(\sigma)\frac{1}{\sqrt{\psi(\sigma,\sigma)}}\right)
=\mathbf{E}(\mathcal{X}_{\sigma}\mathcal{X}_s).
$$
where $\mathcal{X}$ is the centered Gaussian process defined by $
\mathcal{X}_t=\frac{1}{2}\xi_tb(t)\frac{1}{\sqrt{\psi(s,s)}}
$.  The process $(\xi_t,0\leq t \leq T)$ satisfies the following dynamic
$$
- d\xi_t=-\frac{\vartheta}{2}\xi_tA(t)d\langle M\rangle_t+\ell(t)^*\frac{1}{\sqrt{\psi(t,t)}} \odot dW_t,\, \xi_T=0.
$$
Obviously, $\mathcal{K}_T(s,\sigma)$ is a compact symmetric operator for fixed $T$, so we should estimate the spectral gap (the first eigenvalue $\nu_1(T)$) of the operator. The estimation of the spectral gap is based on the Laplace transform computation.
Let us compute, for sufficiently small negative $a<0$ the Laplace transform of $\int_0^T\mathcal{X}_t^2dt$:
\begin{eqnarray*}
  L_T(a) &=& \mathbf{E}_{\vartheta}\exp \left(-a\int_0^T\mathcal{X}_t^2dt\right) \\
   &=&  \mathbf{E}_{\vartheta}\exp \left( -a\int_0^T \left(\frac{1}{2}\xi_tb(t)\frac{1}{\sqrt{\psi(t,t)}}          \right)^2dt\right)
\end{eqnarray*}
On one hand, for $a>- \frac{1}{\nu_1(T)}$, since $\mathcal{X}$ is a centered Gaussian process with covariance operator $\mathcal{K}_T$, using Mercer's theorem and Parseval's inequality, $L_T(a)$ can be represented as :
\begin{equation}\label{represent of Lap tran}
L_T(a)=\prod_{i\geq 1}(1+2a\nu_i(T))^{- \frac{1}{2}},
\end{equation}
where $\nu_i(T),\, i\geq 1$ is the sequence of positive eigenvalues of the covariance operator. On the other hand,
\begin{eqnarray*}
  L_T(a) &=& \mathbf{E}_{\vartheta}\left(- \frac{a}{4}\int_0^T\xi_t b(t)b(t)^* \xi_t^*d\langle M\rangle_t\right) \\
   &=& \exp\left( \frac{1}{2}\int_0^T \mbox{trace}(\mathcal{H}(t)\mathcal{M}(t)d\langle M\rangle_t\right)
\end{eqnarray*}
where $\mathcal{M}(t)=\ell(t)^*\ell(t)$ and $\mathcal{H}(t)$ is the solution of Ricatti differential equation:
$$
\frac{d\mathcal{H}(t)}{d\langle M\rangle_t}=\mathcal{H}(t)\mathcal{A}(t)^*+\mathcal{A}(t)\mathcal{H}+\mathcal{H}(t)\mathcal{M}(t)\mathcal{H}(t)-\frac{a}{2}b(t)b(t)^*,
$$
with $\mathcal{A}(t)=-\frac{\vartheta}{2}A(t)$ and the initial condition $\mathcal{H}(0)=\mathbf{0}_{2\times 2}$, provided that the solution of this equation exists for any $0\leq t\leq T$.

It is well know that if $\det \Psi_1(t)>0$, for any $t\in [0,T]$, then $\mathcal{H}(t)=\Psi_1^{-1}(t)\Psi_2(t)$, where the pair of $2\times 2$ matrices $(\Psi_1,\,\Psi_2)$ satisfies the system of linear differential equations:
\begin{equation}\label{defPsi} 
\begin{array}{ll}
  \displaystyle \frac{d\Psi_1(t)}{d\langle M\rangle_t} = -\Psi_1(t)\mathcal{A}(t)- \Psi_2(t) \mathcal{M}(t), & \,\Psi_1(0)=\mathbf{Id}_{2\times 2}, \\
   & \\
 \displaystyle  \frac{d\Psi_2(t)}{d\langle M\rangle_t} = - \frac{a}{2} \Psi_1(t)b(t)b(t)^*+\Psi_2(t)\mathcal{A}(t)^*, & \, \Psi_2(0)=\mathbf{0}_{2\times 2}
  \end{array}
\end{equation}
and
\begin{equation}\label{eq: lap trans}
L_T(a)=\exp\left(- \frac{1}{2}\int_0^T \mbox{trace}\left( \mathcal{A}(t)\right)d\langle N\rangle_t\right)(\det \Psi_{1}(T))^{-\frac{1}{2}}.
\end{equation}
Rewriting the system \eqref{defPsi} in the following form
\begin{equation}\label{changePsi}
\frac{d(\Psi_1(t),\,\Psi_2(t)\mathbf{J})}{d\langle M\rangle_t}=(\Psi_1(t),\,\Psi_2(t)\mathbf{J})\cdot (\Upsilon \otimes A(t)),
\end{equation}
where
$
\mathbf{J}=\left(
      \begin{array}{cc}
        0 & 1 \\
        1 & 0 \\
      \end{array}
    \right) $    and  $\Upsilon=\left(\begin{array}{cc}\frac{\vartheta}{2} & -\frac{a}{2} \\- 1 & -\frac{\vartheta}{2}\end{array}\right)
$

When $-\frac{\vartheta^2}{2}\leq a \leq 0$, we have two real eigenvalues  of the matrix $\Upsilon$, we denote them $(x_i)_{i=1,2}$. It can be checked that there exists a constant $C>0$ such that 
$$
\det \Psi_1(T)=\exp\left( (x_1)T\right)(C+\underset{T\rightarrow \infty}{O}(\frac{1}{T}))
$$
where $x_1=\sqrt{\frac{\vartheta^2}{4}+\frac{a}{2}}$. Therefore, due to the \eqref{eq: lap trans}, we have $\prod_{i\geq 1}(1+2a\nu_i(T))>0$ for any $a>-\frac{\vartheta^2}{2}$. It means that 
$$
\nu_1(T)\leq \frac{1}{\vartheta^2}
$$
\end{proof}

\begin{lem}\label{Lapalace asymptotical}
For $v(t)=v_{opt}(t)$ defined in Lemma \ref{lemma1}, the Laplace Transform 
\begin{equation}\label{MLE: Laplace Transform}
\mathcal{L}_T(\mu)=\mathbf{E}_{\vartheta}\exp \left(-\frac{\mu}{T}\int_0^TQ_t^2d\langle M\rangle_t\right)\xrightarrow[T\to \infty]{}  \exp\left(-\mu\left(\frac{1}{2\vartheta}+\frac{1}{\vartheta^2}\right)\right)
\end{equation}
for every $\mu>0$.
\end{lem}
\begin{proof}
First, we replace $Q_t$ with $\zeta_t$ and rewrite the Laplace transform, that is 
$$
\mathcal{L}_T(\mu)=\mathbf{E}_{\vartheta}\exp \left\{-\frac{\mu}{T}\int_0^T\zeta_t R(t)\zeta_t^*\ d\langle M\rangle_t\right\}
$$
where $\zeta_t$ is defined in \eqref{eq:zeta} and $R(t)=\frac{1}{4}\left(\begin{array}{cc}\psi^2(t,t) & \psi(t,t) \\\psi(t,t) & 1\end{array}\right)$. Following from \cite{Klep01}, we have 
$$
\mathcal{L}_T(\mu)=\exp \left\{-\frac{\mu}{T}\int_0^T\left[\mbox{tr}(\Gamma(t)R(t))+Z^*(t)R(t)Z(t)           \right]d\langle M\rangle_t \right\}
$$
where
$$
\frac{d\Gamma(t)}{d\langle M\rangle_t}=-\frac{\vartheta}{2}A(t)\Gamma(t)-\frac{\vartheta}{2}\Gamma(t)A(t)^*+b(t)b(t)^*-\frac{2\mu}{T}\Gamma(t)R(t)\Gamma(t)
$$
and
\begin{equation}\label{eq: function Z}
Z(t)=\mathbf{E}_{\vartheta}\zeta_t-\frac{\mu}{T}\int_0^t\varphi(t)\varphi^{- 1}\Gamma(s)R(s)Z(s)d\langle M\rangle_s
\end{equation}
with
$$
\frac{d\varphi(t)}{d\langle M\rangle_t}=-\frac{\vartheta}{2}A(t)\varphi(t).
$$
From \cite{CK19} we know that 
$$
\lim_{T\rightarrow \infty}\exp\left(-\frac{\mu}{T}\int_0^T (\tr (\Gamma(t)R(t))) d\langle M\rangle_t=\exp\left(\frac{\mu}{2\vartheta}\right)        \right)
$$
On the other hand we know $\mathbf{E}\zeta_t=\mathcal{P}(t)$ defined in Lemma \ref{lemma1} with $v(t)=v_{opt}(t)$, thus 
$$
\lim_{T\rightarrow \infty}\exp\left(-\frac{\mu}{T} \mathbf{E}\zeta_t R(t)  (\mathbf{E}\zeta_t)^*   \right)=\lim_{T\rightarrow \infty}\exp\left(-\frac{\mu}{4T} \int_0^T (\ell^*(t)\mathcal{P}(t))^2d\langle M\rangle_t    \right)=\exp\left(-\frac{\mu}{\vartheta^2}   \right)
$$
Now, the conclusion is true provided that  
$$
\lim_{T\rightarrow \infty}\left(\frac{\mu}{T}\int_0^t\varphi(t)\varphi^{- 1}\Gamma(s)R(s)Z(s)d\langle M\rangle_s\right)R(t)  \left(\frac{\mu}{T}\int_0^t\varphi(t)\varphi^{- 1}\Gamma(s)R(s)Z(s)d\langle M\rangle_s\right)^*=0.
$$
On one hand, from \cite{BKP12} and \cite{CK19} when $t$ is large enough
\begin{equation}\label{eq: limit F}
\int_0^t |F(t,s)|ds=\left|\frac{\mu}{T}\int_0^t  \varphi(t)\varphi^{- 1}\Gamma(s)R(s)\right|=O\left(\frac{1}{T} \right),\, T\rightarrow \infty
\end{equation}
where 
$$
F(t,s)=\left|\frac{\mu}{T}\varphi(t)\varphi^{- 1}\Gamma(s)R(s)\right|
$$
and
$|\cdot|$ denotes $L^1$ norm of the vector. On the other hand, If we define  the operator $S$ by
$$
S(f)(t)=\int_0^t \int_0^t |F(t,s)|f(s)ds
$$
then equation \eqref{eq: function Z}  leads to 
$$
|Z(t)|\leq  |\mathcal{P}(t)|  +S(|Z|)(t) 
$$ 
or we can say  $(I-S)(|Z|)(t)\leq |\mathcal{P}(t)|\leq Const$. From Equation \eqref{eq: limit F} we have for t and T large enough   
\begin{equation}\label{eq: Z constant}
|Z(t)|\leq (I-S)^{-1}(Const)(t)=\prod_{n=1}^{\infty}S^n(Const.)(t)\leq Const.
\end{equation}
The $Const.$ means some constant, but in different equation they may be different. Combining \eqref{eq: limit F} and \eqref{eq: Z constant} we have for t large enough 
$$
\int_0^t |F(t,s)| |Z(s)| =O\left(\frac{1}{T}\right),\, T\rightarrow \infty
$$
which achieves the proof.

\end{proof}

\begin{lem}\label{lem Laplace Transform}
For $H>1/2$, the limit of the Laplace Transform is 
$$
\lim_{T\rightarrow \infty}\mathcal{L}^{\alpha}_T(\mu)=\lim_{T\rightarrow \infty} \mathbf{E}_{\vartheta}\exp\left(-\frac{\mu}{T}\int_0^T (Q_t^{\alpha})^2 d\langle M\rangle_t\right)=\exp \left(-\frac{\mu}{2\vartheta}  \right),\, \forall \mu>0
$$
and for $H<1/2$,
$$
\lim_{T\rightarrow \infty}\mathcal{L}^{\alpha}_T(\mu)=\lim_{T\rightarrow \infty} \mathbf{E}_{\vartheta}\exp\left(-\frac{\mu}{T}\int_0^T (Q^{\alpha}_t)^2d\langle M\rangle_t\right)=\exp\left(-\mu\left(\frac{1}{2\vartheta}+\left(\frac{\alpha}{\vartheta}  \right)^2\right)   \right),\, \forall \mu>0.
$$
\end{lem} 
\begin{proof}
The same as the optimal input case, when $u(t)=\alpha$, our two dimensional observed process $\zeta^{\alpha}=(\zeta^{\alpha}_t,\, 0\leq t \leq T)$ satisfies the following equation:
\begin{equation}\label{eq zeta alpha}
d\zeta^{\alpha}_t=\alpha b(t)d\langle M\rangle_t-\frac{\vartheta}{2}A(t)\zeta^{\alpha}_td\langle M\rangle_t+b(t)dM_t
\end{equation}
where $A(t)$, $b(t)$ are defined in \eqref{matrix A b}. From the previous proof we know 
\begin{equation}\label{eq Laplace 2 alpha}
\mathcal{L}^{\alpha}_T(\mu)=\exp \left\{-\frac{\mu}{T}\int_0^T\left[\tr(\Gamma(t)R(t))+\mathcal{Z}^*(t)R(t)\mathcal{Z}(t)           \right]d\langle M\rangle_t \right\}
\end{equation}
where 
\begin{equation}\label{eq: function Z alpha}
\mathcal{Z}(t)=\mathbf{E}\zeta^{\alpha}_t-\frac{\mu}{T}\int_0^t\varphi(t)\varphi^{- 1}\Gamma(s)R(s)\mathcal{Z}(s)d\langle M\rangle_s.
\end{equation}
The functions $\Gamma(t)$, $\varphi(t)$  and the matrix $R(t)$ are defined in the previous Lemma. Let us recall that 
$$
\mathbf{E}Q^{\alpha}_t=\mathbf{E}\frac{d}{d\langle M\rangle_t}\int_0^t g(s,t)X^{\alpha}_sds=\frac{d}{d\langle M\rangle_t}\int_0^t g(s,t)\mathbf{E}X^{\alpha}_sds.
$$
When 
$$
dX^{\alpha}_t=(\alpha-\vartheta X_t)dt+d\xi_t
$$
we have 
$$
\mathbf{E}X^{\alpha}_t=\frac{\alpha}{\vartheta}-\frac{\alpha}{\vartheta}e^{-\vartheta t}.
$$
It is obvious that when we calculate the limit of $\frac{1}{T}\int_0^T (\mathbf{E}Q^{\alpha}_t)^2 d\langle M\rangle_t$, the term $\frac{\alpha}{\vartheta}e^{-\vartheta t}$ has no contribution and will be $0$. Now 
$$
\lim_{T\rightarrow \infty}\frac{1}{T}\int_0^T (\mathbf{E}Q^{\alpha}_t)^2 d\langle M\rangle_t=\lim_{T\rightarrow \infty}\left(\frac{\alpha}{\vartheta}\right)^2\frac{1}{T}\langle M\rangle_T.
$$
From \cite{CK19}, this limit will be $0$ when $H>1/2$ and $\left(\frac{\alpha}{\beta}\right)^2$ when $H<1/2$. When 
$$
\int_0^T (\mathbf{E}\zeta^{\alpha}_t R(t) ) (\mathbf{E}\zeta^{\alpha}_t)^* d\langle M\rangle_t=\int_0^T \left(\frac{1}{2} \ell^*(t) \mathbf{E}\zeta^{\alpha}_t R(t) \right)^2d\langle M\rangle_t=\int_0^T (\mathbf{E}Q^{\alpha}_t)^2d\langle M\rangle_t, 
$$
the conclusion is true provided that  
$$
\lim_{T\rightarrow \infty}\left(\frac{\mu}{T}\int_0^t\varphi(t)\varphi^{- 1}\Gamma(s)R(s)\mathcal{Z}(s)d\langle M\rangle_s\right)R(t)  \left(\frac{\mu}{T}\int_0^t\varphi(t)\varphi^{- 1}\Gamma(s)R(s)\mathcal{Z}(s)d\langle M\rangle_s\right)^*=0
$$
and this proof can be also found in the previous Lemma.

\end{proof}

\begin{lem}\label{Laplace log}
For the controlled mixed fractional Ornstein-Uhlenbeck process with the drift parameter $\vartheta$, we have the following limit:
$$
\mathcal{K}_T(\mu)=-\frac{\mu}{T}\log \mathbf{E} \exp\left(-\mu \int_0^T Q_t^2d\langle M\rangle_t         \right) \rightarrow \frac{\mu}{\vartheta^2}+\frac{\vartheta}{2}-\sqrt{\frac{\vartheta^2}{4}+\frac{\mu}{2}},\,\, T\rightarrow \infty.
$$
for all $\mu>-\frac{\vartheta^2}{2}$.
\end{lem}
\begin{proof}
This proof is directly from \cite{Maru} and Lemma \ref{Lapalace asymptotical} or more specially, 
the term $\frac{\vartheta}{2}-\sqrt{\frac{\vartheta^2}{4}+\frac{\mu}{2}}$ comes from \cite{Maru} and $\frac{1}{\vartheta^2}$ from Lemma \ref{Lapalace asymptotical}.
\end{proof}

\end{document}